\documentclass{article}
\usepackage{amsmath,amssymb,amsthm}
\usepackage{color}
\usepackage{ascmac}
\theoremstyle{plain}
\theoremstyle{plain}
\newtheorem{theorem}{Theorem}[section]
\newtheorem{proposition}{Proposition}[section]
\numberwithin{equation}{section}

\newtheorem{corollary}{Corollary}[section]

\newtheorem{Remark}{Remark}
\newtheorem{example}{Example}
\newtheorem{lemma}{Lemma}[section]

\newtheorem*{proposition*}{Proposition} 

\usepackage{authblk}
\usepackage{hyperref}
\usepackage{hyperref}

\title{The ALM-Framework in the Theory of Multivariate Operator Means}
\author{Dante Hoshina and Shuhei Wada}
\date{}

\begin{document}
\maketitle
\begin{abstract}
In this paper, we generalize the ALM-procedure introduced by Ando, Li, and Mathias for extending operator geometric means to multiple variables.  
We prove that the generalized procedure preserves all the properties required by the axioms of operator means, along with several additional desirable properties.  
Moreover, we examine the self-adjointness of the canonical geometric mean constructed through this procedure.
\end{abstract}

\section{Introduction}
The theory of means for two positive operators on a Hilbert space began with the study of the geometric mean, introduced by Pusz and Woronowicz~\cite{PW}, and was later axiomatized in the work of Kubo and Ando~\cite{KA}.
Their theory demonstrated that two-variable operator means can be identified with certain families of functions. 
Consequently, when restricted to the two-variable case, the theory of means becomes essentially a function-theoretic study. 
This enables the analysis of operator inequalities without being bothered by the complexities of non-commutativity.

As for means of three or more variables, it took time before active research started, except for simple examples such as the arithmetic mean. 
A pioneering work in the early study of multivariable means is that of Ando, Li, and Mathias~\cite{ALM}, who not only developed a three-variable geometric mean for matrices, but also suggested a general method for constructing an \( n \)-variable mean from \((n-1)\)-variable ones. 
Their theory has stimulated much subsequent work, leading to the development of other approaches to 
geometric means~(\cite{BMP},\cite{N},\cite{M} ) and to the investigation of their properties~\cite{HSW}.

Meanwhile, the construction method proposed by Ando, Li, and Mathias—the ALM-procedure—has 
come back into focus
through a conjecture by Petz and Temesi, which encouraged further studies by P\'alfia~\cite{P} and 
Uchiyama~\cite{U,U2}. 

\begin{proposition*}(\cite[Theorem 2.4]{U2})
Let $A,B,C$ be positive operators on a Hilbert space and let $\sigma$ be 
a symmetric operator mean in the sense of Kubo and Ando. 
Then  the sequences $A_n, B_n,C_n$ 
defined by $(A_0,B_0,C_0):=(A,B,C)$ and 
\[
(A_{n+1},B_{n+1},C_{n+1}):=(B_n \sigma C_n, C_n \sigma A_n, A_n \sigma B_n) 
\]
converge weakly  to the same limit $S$.  Moreover, 
$${{A_n+B_n+C_n}\over 3}\downarrow S.$$
\end{proposition*}

\bigskip
That is, using the ALM-procedure, a new three-variable mean 
can be obtained from any two-variable symmetric operator mean.  
As a generalization of this fact, we show that a new three-variable mean $M_{\sigma_1,\sigma_2,\sigma_3}$ 
can be constructed from a certain triple of two-variable operator means $(\sigma_1,\sigma_2,\sigma_3)$.
We also obtain some properties of the three-variable mean $M_{\sigma_1,\sigma_2,\sigma_3}$.  
Finally, we show that from a set of \(n\)-variable operator means 
satisfying certain conditions, an \((n+1)\)-variable operator mean can be constructed using the ALM-procedure.

At the end of this section, we describe the structure of this paper.

In Section 2, we review the classical theory of two-variable operator means developed by Kubo and Ando, and we explain some properties of the means that come from this theory. 

In Section 3, we introduce a theory of multivariate operator means. 
As a natural extension of the two-variable theory of Kubo and Ando, we propose a set of axioms that define 
what a multivariate mean for three or more positive operators should be, even when the operators are not 
necessarily invertible, and present several results based on this axiomatic framework.

In Section 4, we prove one of our main theorems, which shows that the ALM-procedure works under certain conditions and gives a three-variable operator mean. The Perron–Frobenius theory is one of the key tools in the proof of this theorem. 
We also show that the canonical geometric mean obtained from the ALM-procedure is self-adjoint.

In Section 5, we show that the ALM-procedure gives multivariate operator means. The proof of this result is a formal extension of the main theorem in the previous section.

\section{Two-variable  operator means}
In this paper, we discuss  a method of constructing operator means 
 involving three or more variables.  Before this, we review the theory of 
 two-variable  operator means. Let \( B(\mathcal{H}) \) denote the set of all bounded linear operators on a Hilbert space \( \mathcal{H} \), and let \( B(\mathcal{H})_+ \) represent the set of all positive operators on \( \mathcal{H} \).  
When a positive operator \( A \) is invertible, we write \( A > 0 \), and define 
$
\mathbb{P}:= \{ A \in B(\mathcal{H}) \mid A > 0 \}.
$ 
For self-adjoint operators \( A \) and \( B \), we write \( A \ge B \) if \( A - B \in B(\mathcal{H})_+ \).
We say that a binary operation \(\sigma\) on 
 \( B(\mathcal{H})_+ \)
 is called an \emph{operator mean}
 if it satisfies the following four axioms \cite{KA}: 
\begin{itemize}
    \item \textbf{Monotonicity:} If \(A \le C\) and \(B \le D\), then \(A \sigma B \le C \sigma D\).
    \item \textbf{Transformer inequality:} For any positive operator \(T\), we have 
    \[
    T (A \sigma B) T \le (TAT) \sigma (TBT).
    \]
    \item \textbf{Downward continuity:} If \(A_n \downarrow A\) and \(B_n \downarrow B\), then 
    \[
    A_n \sigma B_n \downarrow A \sigma B.
    \]
    \item \textbf{Normalized condition:} \(I \sigma I = I\), where \(I\) is the identity operator.
\end{itemize}

Let $OM_+^1$ be the set of all normalized positive operator monotone functions on $(0,\infty)$.  
By the theory of two-variable  operator means, there exists an affine order isomorphism from the set of all operator means onto \( OM_+^1 \). More precisely, for each operator mean \(\sigma\), there exists a unique function \( f_\sigma \in OM_+^1 \) such that
\[
A \sigma B = A^{1/2} f_\sigma(A^{-1/2} B A^{-1/2}) A^{1/2}\quad (A>0,\ B\ge 0).
\]
Here, the function \( f_\sigma(t) (:= 1 \sigma t ) \) is called the \emph{representing function} of \(\sigma\). Therefore, the behavior of a two-variable  operator mean is completely determined by its representing function \( f_\sigma \).

Henceforth, we denote by \( \#_r \) and \( \nabla_r \) the operator means whose representing functions are the power function \( t \mapsto t^r \) and the linear function \( t \mapsto (1 - r) + r t \), respectively, where \( 0 \le r \le 1 \). An operator mean \(\sigma\) is said to be \emph{arithmetic} or \emph{geometric} if \(\sigma = \nabla_\alpha\) or \(\sigma = \#_\alpha\), respectively, for some \(\alpha \in [0,1]\).

To prepare for the following discussion, we introduce some notation and properties 
concerning two-variable  operator means.
First, we define the operator means denoted by $r$ and $l$ as follows:
$$A r B:=B,\quad A l B:= A.$$
An operator mean \(\sigma\) is said to be \emph{non-trivial} if \(\sigma \notin \{l, r\}\). 
It folows from the theory of operator monotone functions that 
for every non-trivial operator mean $\sigma$, 
the derivative  
${{d(1\sigma t)}\over {dt}} ~\big|_{t=1}$ lies in $(0,1)$. 
We refer to this constant as the \emph{weight} of  $\sigma$.

The following property of a non-arithmetic operator mean is called \emph{strict concavity}.
The property of strict concavity plays a crucial role in the proof of the main theorem in this paper.

\begin{proposition}\label{strict concavity1}
Let $\alpha,\beta \in [0,\infty)$,  
and let  
$\sigma$ be a non-trivial operator mean with weight $r\in (0,1)$, which is not 
arithmetic. If 
$$ \alpha\sigma \beta =(1-r)\alpha + r \beta ,$$
then $\alpha = \beta$. 
\end{proposition}

\begin{proof}
It suffices to assume $(\alpha,\beta)\not=(0,0)$.

Let us consider the case where $\alpha \not=0$.  
It follows from the assumption that 
$f_\sigma (\beta / \alpha )=(1-r) + r(\beta /\alpha )$. 
Since $\sigma$ is not arithmetic, we have $(\beta /\alpha ) =1$. 

We next show the case if $\beta\not=0$. 
Let us denote the representation function of the transpose of $\sigma$ by 
$f_\sigma^\circ$. Then we have    
$f_\sigma^\circ (\alpha/\beta )=(1-r)(\alpha/\beta ) + r$, which implies $(\alpha/\beta )=1$.  
\end{proof}
{\Remark 
A two-variable  operator mean is not arithmetic if and only if it is strictly concave.
}

\begin{proposition}\label{transfer2}
Let $\sigma$ be an operator mean and 
$A,B\ge 0$. Then 
$$\langle A\sigma B x~|~x\rangle \le 
\langle A x~|~x\rangle\sigma \langle B x~|~x\rangle 
$$
holds for all $x\in {\cal H}$. 
\end{proposition}
\begin{proof}
It suffices to assume $\|x\|=1$. 
Let $P$ be a rank-$1$ projection defined by $Py:=\langle y ~|~x\rangle x$ for all $y\in {\cal H}$.  
Then, from the transformer inequality, we have  
\begin{align*}
P (A\sigma B) P
&=P \langle A\sigma B x~|~x\rangle P \\
&\le (PAP)\sigma (PBP)\\
&=
  ( \langle A x~|~x\rangle P )
\sigma 
( \langle B x~|~x\rangle P )\\
&=\lim_{e\rightarrow 0}
  ( \langle A x~|~x\rangle P +\epsilon I)
 f_\sigma \left(
{ {  \langle B x~|~x\rangle P +\epsilon I}
 \over {  \langle A x~|~x\rangle P +\epsilon I}}
\right)  \\
&=
\lim_{e\rightarrow 0}\left(
  ( \langle A x~|~x\rangle +\epsilon )
  f_\sigma \left( 
{ {  \langle B x~|~x\rangle  +\epsilon }
 \over {  \langle A x~|~x\rangle  +\epsilon }}
\right) P+ \epsilon (I-P) 
\right)
\\
&=
 \langle A x~|~x\rangle 
\sigma 
 \langle B x~|~x\rangle P.
\end{align*}
\end{proof}

\section{Multivariate operator means}\label{multivariate operator mean}
\subsection{Natural extension}
For $n\ge 3$, a map $M: {(B({\cal H})_+)}^n \rightarrow {B({\cal H})_+}$ is said to be an 
\emph{$n$-variable operator mean} if
$M$ satisfies the following conditions:
\begin{description}
  \item[(I) Monotonicity:]  
  If \( A_j, B_j \in B({\cal H})_+ \) and \( A_j \le B_j \) for all \( j = 1, \dots, n \), then  
  \[
  M(A_1, \dots, A_n) \le M(B_1, \dots, B_n).
  \]

  \item[(II) Transformer inequality:]  
  For any \( A_1, \dots, A_n \in B({\cal H})_+ \), and any $S\in B({\cal H})_+ $, 
  \[
  S M(A_1, \dots, A_n) S \le M(S A_1 S, \dots, S A_n S).
  \]

  \item[(III) Downward continuity:]  
  Suppose \( A_{jk} \in B({\cal H})_+ \) for \( j = 1, \dots, n \) and \( k \in \mathbb{N} \).  
  If for each \( j \), either \( A_{jk} \downarrow A_j \) 
  as \( k \to \infty \), then  
  \[
  M(A_{1k}, \dots, A_{nk}) \downarrow M(A_1, \dots, A_n). 
  \]

  \item[(IV) Normalized condition:]  
  \[
  M(I, \dots, I) = I,
  \]
  where \( I \) denotes the identity operator on \( \mathcal{H} \).
\end{description}

(II) is a natural Extension of the corresponding condition in the definition of two-variable  operator means.  
From this condition, the following one called \emph{congruence invariance} immediately follows:

\begin{description}
  \item[(II)'] \textbf{Congruence invariance:}  
  For any \( A_1, \dots, A_n \in B({\cal H})_+ \) and any \( S \in {\Bbb P} \),
  \[
    S M(A_1, \dots, A_n) S = M(S A_1 S, \dots, S A_n S).
  \]
\end{description}

In the theory of two-variable operator means  \cite{KA}, binary operations for a pair of positive operators are considered.
However, in the multivariate case, it is often the case that the operators involved are assumed to be invertible.
In \cite{HSW}, a map from \( \mathbb{P}^n \) to \( \mathbb{P} \) is said to be a multivariate operator 
mean if it satisfies conditions (I), (II)', (III), and (IV), as well as upward continuity.
Our multivariate operator mean \( M \) satisfies \( M(A_1, \dots, A_n) > 0 \) for any invertible positive operators \( A_1, \dots, A_n \).
However, it is not clear whether the restriction of $M$ to ${\Bbb P}^n$ 
is a multivariate operator mean in the sense of \cite{HSW}.

\bigskip

The following proposition is a consequence of the transformer inequality.
Hereafter, $M$ denotes a three-variable  operator mean.
\begin{proposition} \label{transformer2}
$$\langle M(A,B,C) v~|~v\rangle 
\le 
 M(\langle  A  v ~|~v\rangle , \langle  B  v ~|~v\rangle , \langle  C  v ~|~v\rangle  )$$
holds for all $A,B,C\in B({\cal H})_+$ and for all $v\in {\cal H}$.\end{proposition}
\begin{lemma}
Let $P$ be an orthogonal projection on ${\cal H}$ and let 
$A,B,C\in  B({\cal H})_+$ that commute with $P$. 
Then 
$$M(A,B,C)P=M(AP,BP,CP)P.$$
\end{lemma}
\begin{proof}
We first show that \( M(A,B,C) \) commutes with \( P \).  
From the monotonicity and transformer inequality, we have  
$$PM(A,B,C)P \le M(PAP, PBP, PCP) = M(AP, BP, CP) \le M(A,B,C),
$$
which implies that  
\[
M(A,B,C) - PM(A,B,C)P \ge 0,\]
and 
\[
\left| \left(M(A,B,C) - PM(A,B,C)P\right)^{1/2} P \right|^2 = 0.
\]
Hence, we obtain \( M(A,B,C)P = PM(A,B,C)P \).

On the other hand, the obvious inequality  
\[
M(AP, BP, CP) \ge PM(AP, BP, CP)P
\]
implies that
\[
\left| \left(M(AP, BP, CP) - PM(AP, BP, CP)P\right)^{1/2} P \right|^2 = 0,
\]
and therefore \( M(AP, BP, CP)P = PM(AP, BP, CP)P \).

Thus, both \( M(A,B,C) \) and \( M(AP, BP, CP) \) commute with \( P \).

Multiplying both sides of the inequality
\[
M(A,B,C)P \le M(PAP, PBP, PCP) \le M(A,B,C)
\]
by \( P \) from both sides, we obtain
\[
M(A,B,C)P \le M(PAP, PBP, PCP)P \le M(A,B,C)P.  
\]
Therefore, the equality follows. 
\end{proof}
\begin{proof}[Proof of Proposition \ref{transformer2}]
It is sufficient to  assume $\|v\|=1$. Let $P$ be a rank-$1$ projection defined by 
$Px:=\langle x~|~v\rangle v$ for all $x\in {\cal H}$. 
Since $PAP=\langle Av~|~v\rangle P$, we have 
\begin{align*}
PM(A,B,C)P
&=P \langle M(A,B,C) v~|~v\rangle P \\
&\le 
M(PAP,PBP,PCP) \\
&=
M(\langle Av~|~v\rangle P ,\langle Bv~|~v\rangle P ,\langle Cv~|~v\rangle P). 
\end{align*}
Multiplying both sides by $P$, 
\begin{align*}
\langle M(A,B,C) v~|~v\rangle P 
&\le 
PM(\langle Av~|~v\rangle P ,\langle Bv~|~v\rangle P ,\langle Cv~|~v\rangle P)P \\
&=
M(\langle Av~|~v\rangle   ,\langle Bv~|~v\rangle  ,\langle Cv~|~v\rangle  )P.\\
\end{align*}
The last equality follows from the lemma above.

\end{proof}

\subsection{Adjoint of an operator mean}\label{adjoint of mean}
In this subsection, we discuss an adjoint of a multivariate operator mean.  
Motivated by the definition of the adjoint for two-variable  operator means,  
the following definition is natural.  

Given two multivariate operator means \( M_1 \) and \( M_2 \),  
we say that \( M_2 \) is an \emph{adjoint} of \( M_1 \) if, for any invertible positive operators \( A_1, \dots, A_n \),  
\begin{equation}\label{adjoint}
M_2(A_1, \dots, A_n) =  M_1(A_1^{-1}, \dots, A_n^{-1})^{-1}.
\end{equation}

If $M_2$ and $M_2'$ satisfy (\ref{adjoint}), then by the downward continuity,
\begin{align*}
M_2((A_k))v &= \lim_{\epsilon \to 0} M_2((A_k + \epsilon I))v \\
           &= \lim_{\epsilon \to 0} M_1((A_k + \epsilon I)^{-1})^{-1} v\\
           &= \lim_{\epsilon \to 0} M_2'((A_k + \epsilon I)) v= M_2'((A_k))v,
\end{align*}
hold for all $A_1, \ldots, A_n \in B(\mathcal{H})_+$ and $v\in {\cal H}$.  
This implies that $M_2 = M_2'$.
Therefore, if the adjoint of a multivariate operator mean \( M \) exists, it is unique.  
In what follows, we denote the adjoint of \( M \) by \( M^* \).


\subsection{Weights }
Let \( M \) be an \( n \)-variable operator mean.  
Then the map 
\[
(A, B) \mapsto M(A, B, \ldots, B)
\]
defines a binary operation on \( B(\mathcal{H})_+ \) that satisfies the four axioms of a two-variable  operator mean described above.  
In particular, the function
\[
f_M(t):= M(t, 1, \ldots, 1)
\]
is operator monotone.  

If \( M \) is dominated by an arithmetic mean, that is, there exists a probability vector \( (r_k)_{k=1}^n \subset [0,1]^n \) such that
\begin{equation}\label{dominated}
M(A_1, \ldots, A_n) \le \sum_{k=1}^n r_k A_k
\end{equation}
for all \( A_k \ge 0 \), then \( f_M(t) \le r_1 t + (1 - r_1) \), which implies \( f_M'(1) = r_1 \).  

Similarly, the following proposition holds:

\begin{proposition}\label{derivative}
Let \( M \) be an \( n \)-variable operator mean and let \( (r_k)_{k=1}^n \) be a probability vector.  
If the inequality \eqref{dominated} holds, 
 then for each \( j = 1, \ldots, n \),
\[
\left. \frac{d}{dt} M\big(1, \ldots, 1, \underset{j\text{th}}{t}, 1, \ldots, 1\big) \right|_{t=1} = r_j.
\]
\end{proposition}

\bigskip

To conclude this section, we introduce terminology for \(n\)-variable operator means, following the convention used in the two-variable  case.

Let \(M\) be an \(n\)-variable operator mean. If there exists \(i\) with \(1 \le i \le n\) such that  
\[
M(A_1, \ldots, A_n) = A_i
\]  
holds for all positive operators \(A_1, \ldots, A_n\), then \(M\) is said to be \emph{trivial}.

Moreover, if there exists a probability vector \((r_k)_{k=1}^n\) such that  
\[
M(A_1, \ldots, A_n) = \sum_{k=1}^n r_k A_k
\]  
for all positive operators \(A_1, \ldots, A_n\), then \(M\) is said to be \emph{arithmetic}.

\section{Extension of the ALM-type mean}
\subsection{ALM-procedure}
In the following, we shall discuss a procedure for constructing an 
$n$-variable operator mean from an $(n-1)$-variable one.

The geometric mean for three or more matrices has been discussed in
the literature.  
First, we introduce the method for constructing the three-variable geometric mean proposed by 
Ando, Li and Mathias in \cite{ALM}. For positive definite matrices $A,B,C$, 
we define the sequence $(A_0,B_0,C_0):=(A,B,C)$ and 
\begin{equation}\label{ALM_procedure}
(A_{n+1},B_{n+1},C_{n+1}):=(B_n \# C_n, C_n \# A_n, A_n \# B_n). 
\end{equation}
These sequences $A_n, B_n,C_n$ have the same limit $G(A,B,C)$. 
Moreover, for $n\ge 4$, the n-variable operator mean $G$ 
can also be defined inductively in the same manner. 
The map $(A_1,\ldots, A_n)\to G(A_1,\ldots, A_n)$
 satisfies the conditions described in Section \ref{multivariate operator mean}, 
as well as additional conditions appropriate for a geometric mean \cite{ALM}.

\subsection{Three-variable Extension of an operator mean}

Concerning (\ref{ALM_procedure}), 
Petz and Temesi  \cite{PT} considered the convergence of the following sequence, 
where the geometric mean is replaced by a symmetric operator mean $\sigma$
: 
\begin{equation}\label{extended ALM_procedure}
(A_{n+1},B_{n+1},C_{n+1}):=(B_n \sigma C_n, C_n \sigma A_n, A_n \sigma B_n). 
\end{equation}
\begin{proposition}(\cite[Theorem~1]{PT})
Let $A,B,C$ be positive definite matrices of the same size. If $A\le B\le C$,
then the sequences $A_n, B_n,C_n$ converge to the same limit.
\end{proposition}
The condition $A\le B\le C$  in the preceding proposition is conjectured 
to be unnecessary. An affirmative answer to this conjecture was provided by P\'alfia \cite{P},  and 
it was generalized by Uchiyama \cite{U2} to a proposition for positive operators on an 
infinite-dimensional Hilbert space.

\begin{proposition}(\cite[Theorem 2.4]{U2})
Let $A,B,C$ be in $B({\cal H})_+$ and let $\sigma$ be 
a symmetric operator mean. Then  the sequences $A_n, B_n,C_n$ 
defined recursively as in (\ref{extended ALM_procedure}) 
weakly converge to the same limit $S$.  Moreover, 
$${{A_n+B_n+C_n}\over 3}\downarrow S.$$
\end{proposition}

\subsection{Main result}
Motivated by the preceding result, we investigate whether the assumption that 
the operator mean $\sigma$ is symmetric could be weakened, and obtain the following result. 
Before stating the result, 
we consider the following operator sequence 
constructed from arbitrary 
positive operators $A,B,C\in \mathcal{B}(\mathcal{H})_+$ and operator means 
$\sigma_1,\sigma_2,\sigma_3$:  
\begin{equation}\label{our sequence}
(A_{n+1}, B_{n+1}, C_{n+1}):= (B_n \sigma_1 C_n,\; C_n \sigma_2 A_n,\; A_n \sigma_3 B_n),
\end{equation}
where $(A_0, B_0, C_0):= (A, B, C)$. 
\begin{theorem}\label{main theorem} 
Let 
\( \sigma_1, \sigma_2, \sigma_3 \) 
be non-trivial operator means, and 
suppose that either all of them are arithmetic,  or 
 at most one of which is not strictly concave.
Then there exists a unique probability vector $(p_1,p_2,p_3)\in (0,1)^3$ such that 
for any \( A, B, C \in \mathcal{B}(\mathcal{H})_+ \), 
the sequences \( A_n, B_n, C_n \), defined as in (\ref{our sequence}) 
converge weakly to a common limit \( S \) and 
$${p_1A_n+p_2B_n+p_3C_n}  \downarrow S.$$
\end{theorem}

\bigskip

To show this, we need some lemmas. 
Before stating lemmas, 
we define some terminology. 
We say that a sequence of bounded operators $A_n$ \emph{converges uniformly} to a bounded operator $A$ if 
$\|A_n - A\|$ converges to $0$.
\begin{lemma}\label{lemma1}
If \( \sigma_1, \sigma_2, \sigma_3\) are arithmetic, 
then there exists a unique probability vector $(p_1,p_2,p_3)\in (0,1)^3$ such that 
the sequences \( A_n, B_n, C_n \) converge uniformly to 
$p_1 A+p_2 B +p_3 C$ and
$$ p_1 A_n + p_2 B_n + p_3 C_n \downarrow p_1 A+p_2 B +p_3 C.$$
\end{lemma}
\begin{proof}
The recurrence relation defined in (\ref{our sequence}) 
 can be rewritten as follows: 
$$
\begin{pmatrix}
A_n \\ B_n \\ C_n 
\end{pmatrix}
=
\begin{pmatrix}
 0& 1-r_1  & r_1  \\
 r_2  & 0 &1-r_2  \\
1-r_3  & r_3  & 0
\end{pmatrix}^n
\begin{pmatrix}
A \\ B \\ C 
\end{pmatrix}, 
$$
where $r_k:={{d(1\sigma_k t)}\over {dt}} ~\big|_{t=1}$, \ $(k=1,2,3)$. 
By the Perron–Frobenius theorem for irreducible matrices 
\cite{BR, HJ}, 
there exists a unique probability vector \((p_1, p_2, p_3) \in [0,1]^3\) such that 
\begin{equation}\label{probability vector}
\begin{pmatrix}
p_1 & p_2 &p_3  \\
p_1 & p_2 &p_3  \\
p_1 & p_2 &p_3  
\end{pmatrix}
=
\lim_n 
\begin{pmatrix}
 0& 1-r_1  & r_1  \\
 r_2  & 0 &1-r_2  \\
1-r_3  & r_3  & 0
\end{pmatrix}^n, 
\end{equation}
which implies $p_k\in [0,1]$ and $p_1+p_2+p_3=1$. 
Put 
$$
P:=\begin{pmatrix}
p_1 & p_2 &p_3  \\
p_1 & p_2 &p_3  \\
p_1 & p_2 &p_3  
\end{pmatrix}
 \quad \text{ and  } \quad 
\Gamma:=
\begin{pmatrix}
 0& 1-r_1  & r_1  \\
 r_2  & 0 &1-r_2  \\
1-r_3  & r_3  & 0
\end{pmatrix} .
$$
Then we have 
$P=(\lim_n \Gamma^{n-2} )\  \Gamma^2$. 
Since all entries of $\Gamma^2$ are in $(0,1)$, 
each $p_k$ is in $(0,1)$.  
This concludes the first part of the proof. Next, we show that 
the sequence \( p_1 A_n + p_2 B_n + p_3 C_n \) is constant. It follows from the 
 recurrence relation defined in (\ref{our sequence}) that   
\begin{align*}
&p_{1} A_{n+1} +p_2 B_{n+1} +p_3 C_{n+1}\\
&=
 p_{1} (B_n \nabla_{r_{1}} C_n) +p_2(C_n \nabla_{r_2} A_n) + p_3(A_n \nabla_{r_3} B_n) \\
&=
(p_2r_2+p_3(1-r_3))A_n +
(p_{1}(1-r_{1})+ p_3r_3 )B_n
+(p_{1} r_{1} + p_2 (1-r_2))C_n \\
&=
p_{n+1} A_n +p_2 B_n +p_3 C_n,
\end{align*}
which completes the proof. 
\end{proof}
{\Remark\label{Perron projection} 
Let $p_k\in (0,1)$ and let $r_k \in (0,1)$ $(k=1,2,3)$. 
If $\sum_k p_k=1$, then the matrices $P,\Gamma$ defined as above 
satisfy the following \cite{HJ}: 
$$P\Gamma=P\Leftrightarrow P=\lim_n \Gamma^n.$$
}
{\Remark\label{Perron projection2} 
Let $\Gamma$ be any $n \times n$ (right) stochastic matrix whose diagonal entries are all zero and whose off-diagonal entries are all strictly positive.  
By the Gershgorin circle theorem, all eigenvalues of $\Gamma$ lie within the closed unit disk in the complex plane centered at the origin.  
The Perron--Frobenius eigenvalue of $\Gamma$ is $1$, and all other eigenvalues have modulus strictly less than $1$.  
Therefore, there exists a probability vector $p = (p_1, \ldots, p_n)$ such that
\[
\lim_{m \to \infty} \Gamma^m = (1, \ldots, 1)^t p,
\]
where the limit is taken entrywise. As mentioned in the proof of the previous lemma, each component of $p$ is strictly positive (cf.~\cite{HJ}).

If the assumptions of the above claim are not satisfied, then $\Gamma^m$ may fail to converge. For example, if $\Gamma$ is the $3 \times 3$ cyclic shift matrix, then $\Gamma^m$ does not converge. Furthermore, even if we define 
\[\Gamma:=
\begin{pmatrix} 
0 & 1 & 0\\ 
0 & 0 & 1 \\
1 & 0 & 0 
\end{pmatrix}
\otimes 
\begin{pmatrix} \frac{1}{2} & \frac{1}{2} \\ \frac{1}{2} & \frac{1}{2} \end{pmatrix},
\]
the sequence $\Gamma^m$ still does not converge.
This example is related to Theorem \ref{main theorem2}. 
}

\begin{lemma}\label{lemma2} 
Let 
\( \sigma_1, \sigma_2, \sigma_3 \) 
be non-trivial operator means, at most one of which is not strictly concave.
then there exists a unique probability vector $(p_1,p_2,p_3)\in (0,1)^3$ such that 
the sequences \( A_n, B_n, C_n \) 
converge weakly to a common limit \( S \) and 
$${p_1A_n+p_2B_n+p_3C_n}  \downarrow S.$$
\end{lemma}
\begin{proof}
Assume that $\sigma_1$ and $\sigma_2$ are strictly concave.
Let $r_k$ $(k=1,2,3)$ be the weight of $\sigma_k$ and let 
$(p_1,p_2,p_3)$ be the probability vector defined as in the proof of 
Lemma \ref{lemma1}. 

From the defintition of $A_n,B_n, C_n$ we have 
\begin{align*}
0&\le p_1A_{n+1}+p_2B_{n+1}+p_3C_{n+1} \\
&=p_1 B_n\sigma_1C_n + p_2C_n\sigma_2 A_n + p_3 A_n\sigma_3 B_n \\
&\le p_1(1-r_1)B_n+p_1 r_1 C_n
+
p_2 (1-r_2)C_n+p_2 r_2 A_n \\
&{\color{white}++++++++++++++}+
p_3 (1-r_3)A_n+p_3r_3 B_n \\
&=
  p_1A_n+p_2B_n+p_3C_n ,
\end{align*}
which implies that $S_n:=p_1A_n+p_2B_n+p_3C_n$ is bounded and decreasing sequence, 
so it has the strong limit $S$. 

We next show that for any nonzero vector $v$, 
sequences  $\langle A_n v~|~v\rangle ,\langle B_nv~|~v\rangle$ and 
$\langle C_nv~|~v\rangle$ converge to the same limit $\langle Sv~|~v\rangle$. 
If we put $M:=\max\{\|A\|,\|B\|,|C\|\}$, then the set 
$$\{ (\langle A_n v~|~v\rangle ,\langle B_nv~|~v\rangle, \langle C_nv~|~v\rangle) ~|~n\ge 0\}
$$
is contained in a compact set $[0,M]^3(\subset {\Bbb R}^3)$.
Thus, the sequence 
$$t_n:=(\langle A_n v~|~v\rangle ,\langle B_nv~|~v\rangle, \langle C_nv~|~v\rangle)$$  
has a convergent subsequence $t_{\varphi(n)}$, and its limit $(\alpha,\beta,\gamma)$ 
lies in $[0,M]^3$.  

It follows from the monotonicity of $S_n$ and Proposition \ref{transfer2} that 
\begin{align*}
&\langle (p_1A_{\varphi(n+1) }+p_2B_{\varphi(n+1) }+p_3C_{\varphi(n+1) })v~|~v\rangle \\
&\le 
\langle (p_1A_{\varphi(n)+1 }+p_2B_{\varphi(n)+1 }+p_3C_{\varphi(n)+1 })v~|~v\rangle \\
&=
p_1\langle A_{\varphi (n)+1}v~|~v\rangle +
p_2\langle B_{\varphi (n)+1}v~|~v\rangle +
p_3\langle C_{\varphi (n)+1}v~|~v\rangle \\
&\le 
p_1\langle B_{\varphi (n)}v~|~v\rangle \sigma_1 
\langle C_{\varphi (n)}v~|~v\rangle \\
&{\color{white}+++}+
p_2\langle C_{\varphi (n)}v~|~v\rangle \sigma_2 
\langle A_{\varphi (n)}v~|~v\rangle \\
&{\color{white}+++}+ 
p_3\langle A_{\varphi (n)}v~|~v\rangle \sigma_3 
\langle B_{\varphi (n)}v~|~v\rangle \\
&\le
p_1(1-r_1) \langle B_{\varphi (n)}v~|~v\rangle + p_1r_1 \langle C_{\varphi (n)}v~|~v\rangle \\ 
&{\color{white}+++}+
p_2(1-r_2) \langle C_{\varphi (n)}v~|~v\rangle +p_2r_2 \langle A_{\varphi (n)}v~|~v\rangle \\
&{\color{white}+++}+
p_3(1-r_3) \langle A_{\varphi (n)}v~|~v\rangle + p_3r_3 \langle B_{\varphi (n)}v~|~v\rangle \\
&=
p_1\langle A_{\varphi (n)}v~|~v\rangle +
p_2\langle B_{\varphi (n)}v~|~v\rangle +
p_3\langle C_{\varphi (n)}v~|~v\rangle .
\end{align*}
Taking the limit of each term, the followings hold
$$p_1\alpha + p_2\beta +p_3\gamma 
\le p_1\beta\sigma_1 \gamma + p_2\gamma \sigma_2 \alpha + p_3\alpha \sigma_3 \beta 
\le p_1\alpha + p_2\beta +p_3\gamma .
$$
Thus 
$$\beta \sigma_1 \gamma = (1-r_1) \beta + r_1 \gamma, \quad \gamma \sigma_2 \alpha=(1-r_2) \gamma +r_2 \alpha . $$
From Proposition \ref{strict concavity1} and  
the assumption that $\sigma_1,\sigma_2$ are strictly concave, 
we have $\alpha=\beta=\gamma$. 
So 
$$
\alpha=\beta=\gamma=p_1\alpha+p_2\beta+p_3\gamma=
\lim_n \langle S_{\varphi(n)} v~|~v\rangle=\langle S v~|~v\rangle.
$$
Thus, every convergent subsequence of $t_n$ has the limit 
$$(\langle S v~|~v\rangle, \langle S v~|~v\rangle, \langle S v~|~v\rangle ).$$
This implies that $t_n$ is a convergent sequence. 

Finally, we show the uniqueness of the probability vector $p:=(p_1,p_2,p_3)$. 
For simplicity, we denote the representation function of $\sigma_k$ by $f_k$. 
For $t>0$, put 
$$
A:=\begin{pmatrix}
t &  &   \\
  & 1 &   \\
  &   &1  
\end{pmatrix}, \ \ 
B:=\begin{pmatrix}
1 &  &   \\
  & t &   \\
  &   &1  
\end{pmatrix}, \ \ 
C:=\begin{pmatrix}
1 &  &   \\
  & 1 &   \\
  &   &t  
\end{pmatrix}. 
$$
Then 
the relation $S_0 \ge S_1$ implies 
$$p_1 t+ p_2+p_3 \ge p_1 + p_2f_2(t) + p_3f_3^\circ(t),$$
$$p_1 + p_2t+p_3\ge p_1 f_1^\circ(t) +p_2 + p_3 f_3(t),$$
and
$$p_1 + p_2+p_3 t \ge p_1 f_1(t) +p_2f_2^\circ(t) + p_3 ,$$
where $f_k^\circ(t):=tf_k(1/t)$. 

The inequality above holds for all $t>0$, 
so by differentiating both sides at $t=1$, 
we obtain the following linear equation:  
\begin{equation}\label{linear}
\left\{ \,
    \begin{aligned}
    & r_2 p_2+(1-r_3 )p_3=p_1 \\
    & (1-r_1 )p_1+r_3 p_3=p_2 \\
    &r_1 p_1+ (1-r_2 )p_2=p_3, 
   \end{aligned}
\right.
\end{equation}
namely,  $p\Gamma=p$. Since the Perron eigenvalue of $\Gamma$ is $1$, 
 the probability vector $p=(p_1,p_2,p_3)$ is unique \cite{HJ}.
 \end{proof}

\begin{proof}[Proof of Theorem \ref{main theorem}]
Immediate from Lemma \ref{lemma1} and Lemma \ref{lemma2}. 
\end{proof}
Theorem \ref{main theorem} determines a unique probability vector \((p_1, p_2, p_3)\) 
corresponding to each admissible triple of two-variable  operator means \((\sigma_1,\sigma_2,\sigma_3)\).
{\color{black}
We refer to this as the probability vector corresponding to $(\sigma_1,\sigma_2,\sigma_3)$, 
and to 
the sequence of triples of operators $(A_n,B_n,C_n)$ defined by (\ref{our sequence}) 
as the \emph{ALM-sequence} from  $(A,B,C)$, with $(\sigma_1,\sigma_2,\sigma_3)$.  
}

\bigskip

{\Remark 
Solving (\ref{linear}), we obtain
$$
\begin{pmatrix}
p_1 \\
p_2\\
p_3
\end{pmatrix}
=
{1\over D}
\begin{pmatrix}
(1-r_3 )+r_2 r_3  \\
(1-r_1 )+r_3 r_1 \\
(1-r_2 )+r_1 r_2 
\end{pmatrix},
$$
where $
D=(1-r_1 )+(1-r_2 )+(1-r_3 )+r_1 r_2 +r_2 r_3 +r_3 r_1. 
$
This makes clear the correspondence 
$(r_1,r_2,r_3) \mapsto (p_1,p_2,p_3)\in (0,1)^3$. 
We denote this map by $\varphi$. 
It follows from repeated simple calculations that 
the map $\varphi: (0,1)^3 \to 
\{(p_1,p_2,p_3)\in (0,1)^3~|~ \sum_k p_k=1\}$ 
is neither injective nor surjective.

\bigskip

}

By the above argument, we can define a map $M_{\sigma_1,\sigma_2,\sigma_3}$
that maps a triple of positive operators $(A,B,C)$ to 
the weak limit of the sequence $A_n$. 
Let us show that the map $M_{\sigma_1,\sigma_2,\sigma_3}$
 is a multivariate (three-variable) operator mean. 
The properties of monotonicity, transformer inequality, and the normalized condition of $M_{\sigma_1,\sigma_2,\sigma_3}$ 
are clear. We now prove downward continuity below. 
Although the proof is almost the same as that in \cite{U2}, 
we prove it here for the reader's convenience.
\begin{proposition}\label{downward}
Let \( A, B, C \geq 0 \) be positive operators, and suppose
\[
A^{(k)} \downarrow A, \quad B^{(k)} \downarrow B, \quad C^{(k)} \downarrow C.
\]
Then,
\[
M_{\sigma_1,\sigma_2,\sigma_3}(A^{(k)},B^{(k)},C^{(k)}) \downarrow 
M_{\sigma_1,\sigma_2,\sigma_3}(A,B,C).
\]
\end{proposition}

\begin{proof}
For simplicity, we denote \( M_{\sigma_1,\sigma_2,\sigma_3} \) by \( M \).  
To show that the bounded monotone sequence \( M(A^{(k)}, B^{(k)}, C^{(k)}) \) converges strongly to \( M(A, B, C) \),  
it suffices to prove that it converges weakly to \( M(A, B, C) \).

{\color{black}
Let \( (A^{(k)}_n, B^{(k)}_n, C^{(k)}_n) \) be 
the ALM-sequence from $ (A^{(k)}, B^{(k)}, C^{(k)}) $, with 
$(\sigma_1,\sigma_2,\sigma_3 )$. }
If we define \( S_n \) and \( S^{(k)}_n \) by
\[
S_n:= p_1 A_n + p_2 B_n + p_3 C_n, \quad 
S^{(k)}_n:= p_1 A^{(k)}_n + p_2 B^{(k)}_n + p_3 C^{(k)}_n,
\]
then \( S^{(k)}_n \downarrow S_n \) as \( k \to \infty \).  
Furthermore,
\[
S^{(k)}_n \downarrow M(A^{(k)}, B^{(k)}, C^{(k)}) \quad (n \to \infty).
\]

Hence, for any \( k, n \), we obtain
\begin{align*}
0 &\leq \langle M(A^{(k)}, B^{(k)}, C^{(k)})x | x \rangle 
    - \langle M(A, B, C)x | x \rangle \\
&\leq \langle S^{(k)}_n x | x \rangle - \langle M(A, B, C)x | x \rangle \\
&= \langle S^{(k)}_n x | x \rangle - \langle S_n x | x \rangle 
    + \langle S_n x | x \rangle - \langle M(A, B, C)x | x \rangle.
\end{align*}

Therefore, by choosing \( n \) sufficiently large, and then \( k \) sufficiently large depending on \( n \),  
we can see that 
\[
\langle M(A^{(k)}, B^{(k)}, C^{(k)})x | x \rangle - \langle M(A, B, C)x | x \rangle
\]
is arbitrarily small. This completes the proof.
\end{proof}

To conclude this section, we present some examples of three-variable operator means.
Here, each \((p_1, p_2, p_3)\) is the probability vector corresponding to the given triple of two-variable  means. 
{\example\label{arithmetic} 
From Lemma \ref{lemma1}, for $A,B,C \ge 0$ and for $\nabla_{r_1},\nabla_{r_2},\nabla_{r_3}$, 
the sequences $A_n,B_n,C_n$ defined as (\ref{our sequence}) uniformly converge to
$p_1A+p_2B+p_3C$, namely, $M_{\nabla_{r_1},\nabla_{r_2},\nabla_{r_3}}(A,B,C)=p_1A+p_2B+p_3C$.
}
{\example 
Although it is difficult to express \(M_{\#_{r_1}, \#_{r_2}, \#_{r_3}}(A,B,C)\) in a single formula in general,
it can be written in a simple form when positive operators $A,B,C$ commute with each other.  
From the congruence invariance and downward continuity of \(M_{\#_{r_1}, \#_{r_2}, \#_{r_3}}\),  
we may assume without loss of generality that \(A, B, C \ge I\).
Put $X:=\log_e A$, $Y:=\log_e B$, $Z:=\log_e C$. Then 
the ALM-sequence $(A_n,B_n,C_n)$ 
from $(A,B,C)$, with $(\#_{r_1}, \#_{r_2}, \#_{r_3})$ 
can be written as follows: 
\begin{align*}
(A_1,B_1,C_1)&=( B\#_{r_1} C,  C\#_{r_2} A,A\#_{r_3} B)\\
&=( e^{Y\nabla_{r_1} Z} ,e^{ Z\nabla_{r_2} X}, e^{X\nabla_{r_3} Y} ),
\end{align*}
$$(A_n,B_n,C_n)=( e^{X_n} ,e^{Y_n}, e^{Z_n} ),$$
{\color{black}
where $(X_n,Y_n,Z_n)$ is the ALM-sequence 
from $(X,Y,Z)$, with $(\nabla_{r_1},\nabla_{r_2},\nabla_{r_3})$.}
Since $X_n, Y_n,Z_n$ converge uniformly to $p_1X+p_2Y+p_3Z$, 
$A_n,B_n,C_n$ converge uniformly to $e^{p_1X+p_2Y+p_3Z}=A^{p_1}B^{p_2}C^{p_3}$. Thus, we have 
$$M_{\#_{r_1}, \#_{r_2}, \#_{r_3}}(A,B,C)=A^{p_1}B^{p_2}C^{p_3},$$
when $A,B,C$ commute with each other. 
}
{\example\label{harmonic} 
Let $r\in [0,1]$.  
 The two-variable  harmonic mean $!_r$ is defined as an operator mean 
 whose representing function is given by $t\mapsto {t\over {r+(1-r)t}}$. 
For a triple of positive invertible operators \((A,B,C)\) and 
two-variable  harmonic means \(!_{r_1}, !_{r_2}, !_{r_3}\),   
let \((A_n, B_n, C_n)\) be {\color{black}
the ALM-sequence from $(A,B,C)$, with $(!_{r_1}, !_{r_2}, !_{r_3})$. }
Then 
\begin{align*}
\begin{pmatrix}
A_n^{-1} \\ B_n^{-1} \\ C_n^{-1} 
\end{pmatrix}
&=
\begin{pmatrix}
 0& 1-r_1 & r_1 \\
 r_2 & 0 &1-r_2 \\
1-r_3 & r_3 & 0
\end{pmatrix}
\begin{pmatrix}
A_{n-1}^{-1} \\ B_{n-1}^{-1} \\ C_{n-1}^{-1} 
\end{pmatrix} \\
&=
\begin{pmatrix}
 0& 1-r_1 & r_1 \\
 r_2 & 0 &1-r_2 \\
1-r_3 & r_3 & 0
\end{pmatrix}^n
\begin{pmatrix}
A^{-1} \\ B^{-1} \\ C^{-1} 
\end{pmatrix} .
\end{align*}
Thus, $A_n^{-1}$ converges  uniformly to $p_1A^{-1}+p_2B^{-1}+p_3C^{-1}$. So, 
$$M_{!_{r_1},!_{r_2},!_{r_3}}(A,B,C)=(p_1A^{-1}+p_2B^{-1}+p_3C^{-1})^{-1}.$$ 
}

\subsection{Properties of  $M_{\sigma_1, \sigma_2, \sigma_3}$ }
Let \((\sigma_1, \sigma_2, \sigma_3)\) and \((\sigma_1', \sigma_2', \sigma_3')\) be two triples of operator means satisfying the assumptions of Theorem \ref{main theorem}. 
In this subsection, we discuss some properties of \(M_{\sigma_1, \sigma_2, \sigma_3}\). 
The following is obvious from the construction of 
\(M_{\sigma_1, \sigma_2, \sigma_3}\).
\begin{proposition} \label{mean order}
If \(\sigma_k \le \sigma_k'\) for \(k = 1, 2, 3\), then 
$$M_{\sigma_1,\sigma_2,\sigma_3}(A,B,C)\le 
M_{\sigma_1',\sigma_2',\sigma_3'}(A,B,C).
$$
\end{proposition}
\begin{corollary} 
Let $\sigma_1,\sigma_2, \sigma_3$ be a triple of operator means satisfying the assumptions of Theorem \ref{main theorem} and let $(p_1,p_2,p_3)$ be the probability vector corresponding to 
$(\sigma_1,\sigma_2, \sigma_3)$. Then 
$$(p_1A^{-1}+p_2B^{-1}+p_3C^{-1})^{-1} \le M_{\sigma_1, \sigma_2, \sigma_3} (A,B,C) \le
p_1A+p_2B+p_3C,$$
for all $A,B,C >0$. 
\end{corollary}
\begin{proof} 
This follows directly from 
from Proposition \ref{mean order} and Examples \ref{arithmetic}, \ref{harmonic}.
\end{proof}

The next result generalizes 
the following inequality for two-variable  operator means \cite{U2}:
if $0 < m \le A, B, A', B' \le M$, then  
$$
\|A \sigma B - A' \sigma B'\| \le {M\over m}\max\{\|A - A'\|, \|B - B'\|\}.
$$

\begin{proposition}
If \(0 < m \le A, A', B, B', C, C' \le M\), then
\begin{align*}
&\big\|
M_{\sigma_1,\sigma_2,\sigma_3}(A,B,C)
-
M_{\sigma_1,\sigma_2,\sigma_3}(A',B',C')
\big\| \\
&{\color{white}++++++++}\le \frac{M}{m} \max\{\|A - A'\|, \|B - B'\|, \|C - C'\|\}.
\end{align*}
\end{proposition}

\begin{proof}
Let \(\delta:= \max\{\|A - A'\|, \|B - B'\|, \|C - C'\|\}\). Then,
\[
A' - A \le \delta I \le (\delta/m)A,
\quad \text{so} \quad A' \le (1 + \delta/m)A.
\]
Similarly, \(B' \le (1 + \delta/m)B\) and \(C' \le (1 + \delta/m)C\). Hence,
\[
M_{\sigma_1,\sigma_2,\sigma_3}(A,B,C)
\le (1 + \delta/m) M_{\sigma_1,\sigma_2,\sigma_3}(A',B',C').
\]
It follows that
\begin{align*}
M_{\sigma_1,\sigma_2,\sigma_3}(A,B,C)
-
M_{\sigma_1,\sigma_2,\sigma_3}(A',B',C') 
&\le (\delta/m) M_{\sigma_1,\sigma_2,\sigma_3}(A',B',C') \\
&\le (\delta/m) M.
\end{align*}
A similar argument shows that
\[
M_{\sigma_1,\sigma_2,\sigma_3}(A',B',C')
-
M_{\sigma_1,\sigma_2,\sigma_3}(A,B,C)
\le (\delta/m) M.
\]
\end{proof}

\begin{corollary}\label{numerical continuity}
Let \(\{a(k)\}, \{b(k)\}, \{c(k)\}\) be sequences of positive numbers converging to \(a, b, c > 0\), respectively. Then
\[
M_{\sigma_1,\sigma_2,\sigma_3}(a(k), b(k), c(k)) \to M_{\sigma_1,\sigma_2,\sigma_3}(a, b, c).
\]
\end{corollary}

\begin{proposition} \label{joint homo}
Let $(p_1,p_2,p_3)$ be the probability vector corresponding to 
$(\#_{r_1},\#_{r_2}, \#_{r_3})$. Then 
\[
M_{\#_{r_1},\#_{r_2},\#_{r_3}}(\alpha A, \beta B, \gamma C)
=
\alpha^{p_1} \beta^{p_2} \gamma^{p_3}
M_{\#_{r_1},\#_{r_2},\#_{r_3}}(A, B, C)
\]
holds for all $\alpha, \beta, \gamma > 0$ and for all $A, B, C \ge 0$.  
\end{proposition}
\begin{proof}
Set $(A', B', C'):= (\alpha A, \beta B, \gamma C)$. 
{\color{black}
We denote by \((A_n, B_n, C_n)\) the ALM-sequence from \((A, B, C)\), 
and by \(([A']_n, [B']_n, [C']_n)\) the one from \((A', B', C')\). 
}
Now, we have 
\[
[A']_1 = B' \#_{r_1} C' = \beta^{1 - r_1} \gamma^{r_1} (B \#_{r_1} C) = \beta^{1 - r_1} \gamma^{r_1} A_1,
\]
and  
\[
[B']_1 = \gamma^{1 - r_2} \alpha^{r_2} B_1, \quad 
[C']_1 = \alpha^{1 - r_3} \beta^{r_3} C_1.
\]
Hence,
\begin{align*}
[A']_2 &= [B']_1 \#_{r_1} [C']_1 
= \left( \gamma^{1 - r_2} \alpha^{r_2} \right)^{1 - r_1} \left( \alpha^{1 - r_3} \beta^{r_3} \right)^{r_1} A_2 \\
&= \alpha^{r_2(1 - r_1) + (1 - r_3) r_1} 
\beta^{r_3 r_1} 
\gamma^{(1 - r_2)(1 - r_1)} A_2.
\end{align*}
Iterating this, for every $v\in {\cal H}$, 
\[
\lim_n \langle [A']_n x | x \rangle 
= \alpha^{p_1} \beta^{p_2} \gamma^{p_3} \lim_n \langle A_n x | x \rangle,
\]
which completes the proof.
\end{proof}
Following \cite{T}, 
we recall the definition of the Thompson metric \( d_T(A,B) \) for positive invertible operators \( A \) and \( B \), which is given by
\begin{align*}
d_T(A,B) &:=\log \max\left\{ \ \inf \{ \lambda>0~|~ A\le \lambda B \},\  
\inf \{ \mu>0~|~ B\le \mu A \}\ \right\} \\
&= \| \log (A^{-1/2} B A^{-1/2}) \|.
\end{align*}

\begin{corollary}Let $(p_1,p_2,p_3)$ be the probability vector corresponding to 
$(\#_{r_1},\#_{r_2},\#_{r_3})$. Then 
\begin{align*}
d_T(
M_{\#_{r_1},\#_{r_2},\#_{r_3}}(&A,B,C),
M_{\#_{r_1},\#_{r_2},\#_{r_3}}(A',B',C')
) \\
&\le p_1 d_T(A,A')+
p_2  d_T(B,B')+
p_3 d_T(C,C'),
\end{align*}
holds for all $A,B,C,A',B',C'>0$. 
\end{corollary}
\begin{proof} 
We denote the spectral radius of an operator \( X \) by \( \rho(X) \).  
By the preceding proposition, we have
\begin{align*}
M_{\#_{r_1},\#_{r_2},\#_{r_3}}(&A,B,C)  \\
&\le M_{\#_{r_1},\#_{r_2},\#_{r_3}}(\rho(A'^{-1}A)A',\rho(B'^{-1}B)B', \rho(C'^{-1}C)C') \\
&= \rho(A'^{-1}A)^{p_1}\rho(B'^{-1}B)^{p_2} \rho(C'^{-1}C)^{p_3}M_{\#_{r_1},\#_{r_2},\#_{r_3}}(A',B',C').
\end{align*}
Similarly,
\begin{align*}
M_{\#_{r_1},\#_{r_2},\#_{r_3}}(&A',B',C') \\ 
&\le \rho(A^{-1}A')^{p_1}\rho(B^{-1}B')^{p_2} \rho(C^{-1}C')^{p_3}M_{\#_{r_1},\#_{r_2},\#_{r_3}}(A,B,C).
\end{align*}
Therefore,
\begin{align*}
d_T\big(M_{\#_{r_1},\#_{r_2},\#_{r_3}}(A,&B,C),\ M_{\#_{r_1},\#_{r_2},\#_{r_3}}(A',B',C')\big) \\
&\le p_1\, d_T(A,A') + p_2\, d_T(B,B') + p_3\, d_T(C,C').
\end{align*}
\end{proof}


As mentioned earlier, the map from a given triple of means to a probability vector in the construction of \( M_{\sigma_1, \sigma_2, \sigma_3} \) is not injective.  
For example, the triple \( (\#, \#, \#) \) corresponds to the probability vector \( \left( \tfrac{1}{3}, \tfrac{1}{3}, \tfrac{1}{3} \right) \),  
and so does the triple \( \left( \#_{2/3}, \#_{2/3}, \#_{2/3} \right) \).  
Therefore, for commuting operators \( A, B, C \), we have
\[
M_{\#, \#, \#}(A, B, C) = M_{\#_{2/3}, \#_{2/3}, \#_{2/3}}(A, B, C) = (ABC)^{1/3}.
\]
The following inequality is intended to quantify the difference between these two means in the general (non-commuting) case.

\begin{proposition}\label{distance}
Let $A,B,C>0$, and suppose that  
$(p_1,p_2,p_3)$ is a probability vector corresponding to both 
$(\#_{r_1}, \#_{r_2}, \#_{r_3})$  and $(\#_{r'_1}, \#_{r'_2}, \#_{r'_3})$. 
Then the following inequality holds:
\begin{align*}
d_T(M_{\#_{r_1}, \#_{r_2}, \#_{r_3}}(A,B,C),\ &M_{\#_{r'_1}, \#_{r'_2}, \#_{r'_3}}(A,B,C))\\
&\le 
K 
(p_1 d_T(B,C) + p_2 d_T(C,A) + p_3 d_T(A,B))
,
\end{align*}
where $K=\max\{|r_1-r_1'|,|r_2-r_2'|,|r_3-r_3'|\}$.
\end{proposition}
The following lemma will be needed.  
Here, \( \rho(X) \) denotes the spectral radius of an invertible operator \( X \), and  
\[
R(X,Y) := \max\left\{ \rho(X^{-1}Y),\ \rho(XY^{-1}) \right\}.
\]
\begin{lemma} Let $X,Y>0$. Then
$$R(X\#_{r} Y, X\#_{s} Y)\le R(X,Y)^{|r-s|}.$$
\end{lemma}

\begin{proof}It is sufficient to assume $r\ge s$. It is clear that 
$$X\#_{r} Y\le \rho(X^{-1}Y)^{r-s} X\#_{s} Y,$$
which implies the disired result.
\end{proof}
\begin{proof}[Proof of Proposition \ref{distance}]
Let \( (A_n, B_n, C_n) \) and \( (A'_n, B'_n, C'_n) \) be the ALM-sequences from \( (A, B, C) \) with 
\( (\#_{r_1}, \#_{r_2}, \#_{r_3}) \) and \( (\#_{r'_1}, \#_{r'_2}, \#_{r'_3}) \), respectively.  
From the above lemma, we have
\[
A_1 = B \#_{r_1} C \le R(B, C)^{|r_1 - r_1'|} (B \#_{r_1'} C) = R(B, C)^{|r_1 - r_1'|} A'_1.
\]
Similarly,
\[
B_1 \le R(C, A)^{|r_2 - r_2'|} B'_1, \quad
C_1 \le R(A, B)^{|r_3 - r_3'|} C'_1.
\]
To simplify the notation, let
\[
a := R(B, C)^{|r_1 - r_1'|}, \quad
b := R(C, A)^{|r_2 - r_2'|}, \quad
c := R(A, B)^{|r_3 - r_3'|}.
\]
Then,
\[
A_2 \le (b \#_{r_1} c) A'_2 =: a_1 A'_2, \quad
B_2 \le b_1 B'_2, \quad
C_2 \le c_1 C'_2.
\]
By repeating this process, we obtain
\[
A_n \le a_{n-1} A'_n, \quad
B_n \le b_{n-1} B'_n, \quad
C_n \le c_{n-1} C'_n, 
\]
where $(a_n,b_n,c_n)$ is the ALM-sequence from $(a,b,c)$ with 
$(\#_{r_1}, \#_{r_2}, \#_{r_3})$. 
Taking the limit yields
\[
M_{\#_{r_1}, \#_{r_2}, \#_{r_3}}(A, B, C)
\le (a^{p_1} b^{p_2} c^{p_3}) M_{\#_{r'_1}, \#_{r'_2}, \#_{r'_3}}(A, B, C).
\]
Therefore,
\begin{align*}
&d_T(M_{\#_{r_1}, \#_{r_2}, \#_{r_3}}(A, B, C), M_{\#_{r'_1}, \#_{r'_2}, \#_{r'_3}}(A, B, C)) \\
&\le \max\{|r_1 - r_1'|, |r_2 - r_2'|, |r_3 - r_3'|\} \cdot (p_1 d(B, C) + p_2 d(C, A) + p_3 d(A, B)).
\end{align*}

\end{proof}

\section{Generalization }
\subsection{Strict concavity}
Let \((\sigma_1, \sigma_2, \sigma_3)\) 
be a triple of non-trivial operator means, 
at most one of which is arithmetic.
The next result and its corollary are 
counterparts of Proposition \ref{strict concavity1}.  

\begin{proposition}\label{strict-concavity}
Let \((p_1, p_2, p_3)\) be the probability vector corresponding to \((\sigma_1, \sigma_2, \sigma_3)\) 
and let $A,B,C\in B({\cal H})_+$ and $v \in {\cal H}$.    
If 
$$\langle (M_{\sigma_1,\sigma_2,\sigma_3}(A,B,C) )v~|~v\rangle 
=
\langle (p_1 A + p_2 B + p_3 C) v ~|~v\rangle, $$
then
$\langle  A  v ~|~v\rangle =\langle  B  v ~|~v\rangle=\langle  C  v ~|~v\rangle$.
\end{proposition}

\begin{proof} 
We denote by \((A_n, B_n, C_n)\) the ALM-sequence.  
It suffices to assume that \(\sigma_1\) and \(\sigma_2\) are strictly concave  
and \(v \neq 0\).  
By Proposition~\ref{transfer2}, we have
\begin{align*}
&\langle (p_1A_{n+1}+p_2B_{n+1}+p_3C_{n+1})v | v \rangle \\
&= p_1\langle A_{n+1}v | v \rangle + p_2\langle B_{n+1}v | v \rangle + p_3\langle C_{n+1}v | v \rangle \\
&\le p_1\big( \langle B_n v | v \rangle \,\sigma_1\, \langle C_n v | v \rangle \big)  + p_2\big( \langle C_n v | v \rangle \,\sigma_2\, \langle A_n v | v \rangle \big) \\
&\qquad \qquad \qquad\qquad \qquad \qquad+ p_3\big( \langle A_n v | v \rangle \,\sigma_3\, \langle B_n v | v \rangle \big) \\
&\le p_1(1-r_1)\langle B_n v | v \rangle + p_1 r_1 \langle C_n v | v \rangle 
+ p_2(1 - r_2) \langle C_n v | v \rangle + p_2 r_2 \langle A_n v | v \rangle \\
&\qquad \qquad \qquad\qquad \qquad \qquad\qquad+ p_3(1 - r_3) \langle A_n v | v \rangle + p_3 r_3 \langle B_n v | v \rangle \\
&= p_1\langle A_n v | v \rangle + p_2\langle B_n v | v \rangle + p_3\langle C_n v | v \rangle.
\end{align*}
Let \(\alpha:= \langle A_n v | v \rangle\), \(\beta:= \langle B_n v | v \rangle\), and \(\gamma:= \langle C_n v | v \rangle\).  
Passing to the limit in each term of the above inequalities, we obtain
\[
p_1\alpha + p_2\beta + p_3\gamma \le 
p_1(\beta \,\sigma_1\, \gamma) + p_2(\gamma \,\sigma_2\, \alpha) + p_3(\alpha \,\sigma_3\, \beta)
\le p_1\alpha + p_2\beta + p_3\gamma,
\]
which implies
\[
\beta \,\sigma_1\, \gamma = (1 - r_1)\beta + r_1\gamma, \quad
\gamma \,\sigma_2\, \alpha = (1 - r_2)\gamma + r_2\alpha.
\]
By Proposition~\ref{strict concavity1}, it follows that \(\alpha = \beta = \gamma\).
\end{proof}
\begin{corollary}\label{strict-concavity2}
If scalars $a,b,c\ge 0$ satisfy 
$$M_{\sigma_1,\sigma_2,\sigma_3}(a,b,c) 
=
p_1 a + p_2 b + p_3 c ,$$
then 
$a=b=c$.
\end{corollary}
\subsection{Extension}
Let \( M \) be an \( n \)-variable operator mean.  
We say that \( M \) is \emph{affinely dominated} if there exists a probability vector \((r_k)_{k=0}^{n-1} \in (0,1)^n\) such that
\[
M(A^{(0)}, \ldots, A^{(n-1)}) \le \sum_k r_k A^{(k)}
\]
for all positive operators \(A^{(0)}, \ldots, A^{(n-1)}\).
An affinely dominated operator mean \(M\) is called \emph{strictly concave} if the following condition holds:  
If positive scalars \(a_0, \ldots, a_{n-1}\) satisfy
\[
M(a_0, \ldots, a_{n-1}) = \sum_k r_k a_k,
\]
then it must follow that \(a_0 = \cdots = a_{n-1}\).

{\Remark 
For multivariate operator means, being non-arithmetic and being strictly concave are not equivalent.  
For example, consider the operator mean defined by  
\[
M(A, B, C):= A \nabla (B \# C).
\]  
This mean satisfies the inequality  
\[
M(A, B, C) \le \frac{1}{2}A + \frac{1}{4}B + \frac{1}{4}C,
\]  
but equality holds when \((A, B, C) = (2, 3, 3)\).

On the other hand, there do exist multivariate operator means that are strictly concave. 
As stated in Corollary~\ref{strict-concavity2}, 
the three-variable operator mean 
$M_{\sigma_1,\sigma_2,\sigma_3}$ 
constructed from a suitable triple of operator means \((\sigma_1, \sigma_2, \sigma_3)\) is strictly concave. 

For a probability vector \( (p_1, p_2, p_3) \in (0,1)^3 \), the arithmetic-geometric mean inequality for positive numbers \( a, b, c \) states that
\[
a^{p_1} b^{p_2} c^{p_3} \leq p_1 a + p_2 b + p_3 c,
\]
with equality if and only if \( a = b = c \).  
Therefore, any operator mean that realizes the weighted geometric mean with weights \( (p_1, p_2, p_3) \) must be strictly concave.

}

\bigskip

This subsection presents an extension from an $n$-variable mean to an $n+1$-variable mean.  


For an arbitrary 
positive operators $A^{(0)}, \ldots ,A^{(n)} \in \mathcal{B}(\mathcal{H})_+$ and n-variable 
operator means $M_0,\ldots ,M_{n}$. For $k\in \{0,1,\ldots n\}$, define 
\begin{equation}\label{our sequence2}
A_{m+1}^{(k)}:=M_k(A_m^{(k+1\mod n+1)}, \ldots , A_m^{(k+n \mod n+1)}),
\end{equation}
where $( A_0^{(0)}, \ldots ,A_0^{(n)}  )=(A^{(0)}, \ldots ,A^{(n)}  ) $. 

In our main theorem stated below, we assert that, under certain conditions,  
the sequence \( A_m^{(k)} \) converges weakly to the same operator for all \( k \).

\begin{theorem}\label{main theorem2} 
Let $M_0,\ldots ,M_{n}$ be non-trivial, affinely dominated $n$-variable operator means. 
Suppose that either all of them are arithmetic, or at most one of which is 
not strictly concave. 
Then there exists a unique probability vector $(p_0,\ldots ,p_{n}) \in (0,1)^{n+1}$ such that, 
for any $A^{(0)}, \ldots ,A^{(n)} \in \mathcal{B}(\mathcal{H})_+$, 
the sequences \( A_{m}^{(0)},\ldots ,A_{m}^{(n)} \), defined as in~\eqref{our sequence2}, 
converge weakly to a common limit \( S \), and
\[
\sum_{k=0}^n p_k A_m^{(k)} \downarrow S.
\]

\end{theorem}

\bigskip

Before stating the lemma, we define the map \(\Phi: {\mathbb R}^{n+1} \to M_{n+1}({\mathbb C})\) by
\[
\Phi({\bf x}):= \operatorname{diag}(x_0, \ldots, x_n),
\]
for any real vector \({\bf x} = (x_0, \ldots, x_n)^t\).

\begin{lemma}\label{core}
Let $M_0,\ldots ,M_{n}$ be non-trivial, affinely dominated $n$-variable operator means. 
Suppose that at most one of which is not strictly concave. 
Then there exists a unique probability vector $(p_0,\ldots ,p_{n})\in (0,1)^{n+1}$ such that 
the sequences \( A_{m}^{(0)},\ldots ,A_{m}^{(n)}  \), defined as in (\ref{our sequence2}) 
converge weakly to a common limit \( S \) and 
$$\sum_{k=0}^n p_k A_m^{(k)}  \downarrow S.$$
\end{lemma}
\begin{proof}
From the assumption, we suppose that the following inequality holds for each \(M_k\):
\[
M_k(A_1, \ldots, A_n) \le \sum_{i=1}^n r_i^{(k)} A_i.
\]
Here, each coefficient \(r_i^{(k)}\) lies in \((0, 1)\), and for each \(k\), we have \(\sum_i r_i^{(k)} = 1\).
As stated in Proposition \ref{derivative}, the probability vector 
$(r_1^{(k)}, \ldots ,r_n^{(k)})$ is unique and 
\[
\left. \frac{d}{dt} M_k\big(1, \ldots, 1, \underset{i\text{th}}{t}, 1, \ldots, 1\big) \right|_{t=1} = r_i^{(k)}.
\]

The existence of the probability vector \((p_0, \ldots, p_n)\) stated in this proposition is evident for the same reason as in the proof of Lemma~\ref{lemma2}. Therefore, 
we will show only the uniqueness of the probability vector \((p_0, \ldots, p_n)\). 
The proof is merely a generalization of the corresponding part of Lemma~\ref{lemma2}.

Now, fix \(t > 0\), and for each \(k \in \{0, 1, \ldots, n\}\), define
\[
A^{(k)}:= \Phi\left(U^k {\bf b}\right),
\]
where \(U\) is an \((n+1) \times (n+1)\) cyclic shift matrix 
and \({\bf b}\) is an \((n+1)\)-dimensional vector given by
\[
U:=
\begin{pmatrix}
0 & 0 & \cdots & 0 & 1 \\
1 & 0 & \cdots & 0 & 0 \\
0 & 1 & \ddots & 0 & 0 \\
\vdots & \ddots & \ddots & \ddots & \vdots \\
0 & 0 & \cdots & 1 & 0
\end{pmatrix}
, \quad
{\bf b}:=
\begin{pmatrix}
t \\
1 \\
1 \\
\vdots \\
1 \\
\end{pmatrix}.
\]

Then, the matrix \(A_1^{(k)}\), constructed using the procedure in (\ref{our sequence2}), can be written using the following function \(f_j^{(k)}\) defined by
\[
f_j^{(k)}(t):=
\begin{cases}
M_k(1, \ldots, 1, \underset{j\text{th}}{t}, 1, \ldots, 1) & \text{if } 1 \le j \le n, \\
1 & \text{if } j = 0,
\end{cases}
\]
as follows:
\[
A_1^{(k)} = \Phi\left(U^k {\bf f}^{(k)}\right),
\]
where
\[
{\bf f}^{(k)}:=
\begin{pmatrix}
f_0^{(k)}(t) \\
f_1^{(k)}(t) \\
\vdots \\
f_n^{(k)}(t)
\end{pmatrix}.
\]

From the monotonicity of the sequence 
$(\sum_{k=0}^n p_k A_m^{(k)} )_{m=0}^\infty$, 
we obtain the following:

\[
\sum_{k=0}^n p_k A_1^{(k)} \le \sum_{k=0}^n p_k A_0^{(k)}.
\]
Computing each side yields:
\[
\sum_{k=0}^n p_k A_1^{(k)} 
=
\Phi\left(
\sum_{k=0}^n p_k U^k {\bf f}^{(k)}
\right)
\le
\sum_{k=0}^n p_k A_0^{(k)}
=
\Phi\left(
\sum_{k=0}^n p_k U^k {\bf b}
\right).
\]

Since the right-hand side is a diagonal matrix whose \((i,i)\)-th entries are linear functions of \(t\), the derivatives of both sides at \(t = 1\) should agree.

Computing the derivatives at \(t = 1\), we have:
\[
\left.
\frac{d}{dt}
\left(
\sum_{k=0}^n p_k A_0^{(k)}
\right)\right|_{t=1}
=
\Phi\left(
\sum_{k=0}^n p_k U^k {\bf b}'
\right),
\]
\[
\left.
\frac{d}{dt}
\left(
\sum_{k=0}^n p_k A_1^{(k)}
\right)\right|_{t=1}
=
\Phi\left(
\sum_{k=0}^n p_k U^k \left.\frac{d}{dt} {\bf f}^{(k)} \right|_{t=1}
\right)
=
\Phi\left(
\sum_{k=0}^n p_k U^k {\bf r}^{(k)}
\right),
\]
where
\[
{\bf b}' =
\begin{pmatrix}
1 \\
0 \\
\vdots \\
0
\end{pmatrix}, \quad
{\bf r}^{(k)} =
\begin{pmatrix}
0 \\
r_1^{(k)} \\
\vdots \\
r_n^{(k)}
\end{pmatrix}.
\]
Since the two derivatives above are equal, we obtain the following system of linear equations:
\[
\begin{pmatrix}
p_0 \\
p_1 \\
\vdots \\
p_n
\end{pmatrix}
=
\begin{pmatrix}
0            & r_1^{(0)}   & \cdots & r_n^{(0)}   \\
r_n^{(1)}    & 0           & r_1^{(1)} & \cdots \\
\vdots       &             & \ddots     & \\
r_1^{(n)}    & \cdots      & r_n^{(n)} & 0
\end{pmatrix}^T
\begin{pmatrix}
p_0 \\
p_1 \\
\vdots \\
p_n
\end{pmatrix}.
\]
As mentioned in the proof of Theorem~\ref{main theorem}, such a probability vector \((p_0, \ldots, p_n)\) is uniquely determined.
\end{proof}

\begin{proof}[Proof of Theorem \ref{main theorem2}]
In the above theorem, if all \( M_0, \ldots, M_n \) are arithmetic,  
then the conclusion not only holds obviously, but the sequence 
\( A_{m}^{(k)}  \)
also converges to \( S \) uniformly.  
The proof of this part is essentially the same as that of Lemma~\ref{lemma1}, and is therefore omitted.
In the remaining case, the result follows immediately from the preceding lemma.
\end{proof}
{
\Remark 
In the above theorem, the assumption of being \emph{affinely dominated} is essential.  
For the case \( n = 5 \), consider the definition  
\[
M_k(X_1, \ldots, X_5) := 
\begin{cases}
\frac{1}{2}X_2 + \frac{1}{2}X_3 & \text{if } k \equiv 0 \pmod{2}, \\
\frac{1}{2}X_1 + \frac{1}{2}X_2 & \text{if } k \equiv 1 \pmod{2}.
\end{cases}
\]
Each \( M_k \) is non-trivial, but not affinely dominated. Furthermore,
\[
(A_m^{(0)}, \ldots, A_m^{(5)})^t = \Gamma^m (A_0^{(0)}, \ldots, A_0^{(5)})^t,
\]
where
\[
\Gamma :=
\begin{pmatrix} 
0 & 1 & 0\\ 
0 & 0 & 1 \\
1 & 0 & 0 
\end{pmatrix}
\otimes 
\begin{pmatrix} \frac{1}{2} & \frac{1}{2} \\ \frac{1}{2} & \frac{1}{2} \end{pmatrix}.
\]
As we saw in Remark~\ref{Perron projection}, the sequence \( \Gamma^m \) does not converge.  
Therefore, the sequence \( A_m^{(k)} \) does not converge either.
}

\bigskip

\subsection{Symmetric Operator means}

In what follows, we denote the operator mean constructed in the above theorem by
\[
\mathbb{M}_{(M_0, \ldots, M_n)}.
\]

If \( M \) is permutation invariant, that is, for any permutation \( \pi \), we have
\[
M(A^{(0)}, \ldots, A^{(n-1)}) = M(A^{(\pi(0))}, \ldots, A^{(\pi(n-1))}),
\]
then the operator mean
$
\mathbb{M}_{(M, \ldots, M)}
$
has the following property.

\begin{theorem}\label{ordered}
Let $A^{(0)},\ldots , A^{(n)}$ be positive operators and $M$ 
be an $n$-variable permutation invariant operator mean 
and let $(A_m^{(0)},\ldots , A_m^{n})$
 be the ALM-sequence from 
$(
A^{(0)},\ldots , A^{(n)}
)$, with $(M,\ldots ,M)$. 
If $A^{(0)}\ge \cdots \ge A^{(n-1)}$
, then 
$A_m^{(0)}$ and $A_m^{(n)}$ converge strongly to 
${\Bbb M}_{(M,\ldots ,M)}$.
\end{theorem} 

\begin{proof}
From the assumption, we have 
\begin{align*}
A_1^{(0)} &= M(A_0^{(1)}, A_0^{(2)}, \ldots , A_0^{(n)}) \\
&\le 
M(A_0^{(0)}, A_0^{(2)}, \ldots , A_0^{(n)}) \\
&=
M( A_0^{(2)}, \ldots , A_0^{(n)},A_0^{(0)}) = A_1^{(1)} \\
&\le 
M(A_0^{(1)}, A_0^{(3)}, \ldots , A_0^{(n)},A_0^{(0)})
 \\
&=
M(A_0^{(3)},  \ldots , A_0^{(n)}, A_0^{(0)}, A_0^{(1)})  = A_1^{(2)} \\
&\le \cdots \le A_1^{(n)}. 
\end{align*}
By repeating this process, the following relation:
\[
A_{2m}^{(0)} \ge \cdots  \ge A_{2m}^{(n)}, \quad \text{and} \quad
A_{2m+1}^{(0)} \le \cdots  \le A_{2m+1}^{(n)}
\]
is obtained. Moreover, we have
\[
A_{2m+2}^{(0)}  = M(( A_{2m+1}^{(k)})_{k\not = 0}) 
 = M((M(( A_{2m}^{(i)})_{i\not = k})_{k\not=0}) 
 \le
 A_{2m}^{(0)},
\]
and 
\[
A_{2m+1}^{(0)} 
 \ge
 A_{2m-1}^{(0)}.
\]
Since both sequences  $\{A_{2m}^{(0)}\}$ and $\{A_{2m+1}^{(0)}\}$ are bounded and monotone, they converge strongly to operators $M_0$ and $M_1$, respectively. Thus,
\[
\begin{aligned}
\langle {\Bbb M}_{(M,\ldots ,M)}(  A^{(0)},\ldots , A^{(n)}    )\, v \mid v \rangle 
&= \lim_n \langle A_{2m}^{(0)}\, v \mid v \rangle = \langle M_0\, v \mid v \rangle \\
&= \lim_n \langle A_{2m+1}^{(0)}\, v \mid v \rangle = \langle M_1\, v \mid v \rangle.
\end{aligned}
\]
Therefore, the sequence $A_m^{(0)}$ converges strongly to 
$ {\Bbb M}_{(M,\ldots ,M)}(  A^{(0)},\ldots , A^{(n)})$. 
Similarly, one can show that the sequence $A_m^{(n)}$ also converges strongly to the same limit.
\end{proof}

In the case \( n = 2 \), the above theorem yields the following result. 
This result corresponds to a generalization of the theorem by Petz and Temesi \cite[Theorem 1]{PT}.
\begin{corollary}
Let \( A, B, C \) be positive operators, \( \sigma \) be a symmetric two-variable operator mean, and \( (A_m, B_m, C_m) \) be the ALM sequence generated from \( (A, B, C) \) with \( (\sigma, \sigma, \sigma) \).
If $A\ge B\ge C$, then $A_m,B_m,C_m$ converge strongly to
$
M_{\sigma,\sigma,\sigma}(A, B, C).
$
\end{corollary}
\begin{proof}
The probability vector corresponding to \( (\sigma, \sigma, \sigma) \) is \( (1/3, 1/3, 1/3) \).  
Hence, by Theorem \ref{main theorem2}, the sequence \( (A_m + B_m + C_m)/3 \) converges strongly to  
$M:=
M_{\sigma,\sigma,\sigma}(A, B, C)
$.   
From the preceding theorem, \( A_m \) and \( C_m \) also converge strongly to $M$. 
Therefore, \( B_m \) must also converge strongly to $M$. 
\end{proof}

\bigskip

\subsection{Application }
Finally, we revisit the discussion of the three-variable operator mean developed in the previous sections.
Assume that 
for each $k=1,2,3$, both 
$\sigma_k$ and its adjoint $\sigma_k^*$ 
satisfy the assumption of Theorem \ref{main theorem}. 
The following results hold.  
\begin{proposition}\label{matrix adjoint}
Let \( A, B, C \) be positive definite matrices. Then 
$$M_{\sigma_1^*, \sigma_2^*, \sigma_3^*}(A^{-1}, B^{-1}, C^{-1})
=M_{\sigma_1, \sigma_2, \sigma_3}(A, B, C)^{-1}.$$
\end{proposition}
\begin{proof}
Let \((A_n, B_n, C_n)\) denote the ALM-sequence from \((A, B, C)\), with \((\sigma_1, \sigma_2, \sigma_3)\),  
and \(([A^{-1}]_n, [B^{-1}]_n, [C^{-1}]_n)\) that from \((A^{-1}, B^{-1}, C^{-1})\), with \((\sigma_1^*, \sigma_2^*, \sigma_3^*)\).

By a simple calculation, we have 
$$(A_n^{-1}, B_n^{-1}, C_n^{-1})=([A^{-1}]_n, [B^{-1}]_n, [C^{-1}]_n),$$
which implies  
$$
M_{\sigma_1^*, \sigma_2^*, \sigma_3^*}(A^{-1}, B^{-1}, C^{-1})
=
\lim_n [A^{-1}]_n
=
\lim_n A_n^{-1} = M_{\sigma_1, \sigma_2, \sigma_3}(A, B, C)^{-1}.
$$
\end{proof}
\begin{corollary}
Let \( X,Y,Z \) be positive definite matrices and let $X_n,Y_n,Z_n$ be 
sequences of positive definite matrices. 
If $X_n, Y_n, Z_n$ are increasing and converge to $X,Y,Z$ respectively, 
then 
$$M_{\sigma_1, \sigma_2, \sigma_3}(X_n,Y_n,Z_n)\uparrow 
M_{\sigma_1, \sigma_2, \sigma_3}(X,Y,Z).$$
\end{corollary}
\begin{proof}
It follows immediately from the preceding proposition. 
\end{proof}
The fact $\#_{r}^*=\#_r$ implies the following. 
\begin{corollary}
Let \( A, B, C \) be positive definite matrices. Then 
$$M_{\#_{r_1}, \#_{r_2}, \#_{r_3}}(A^{-1}, B^{-1}, C^{-1})^{-1}
=M_{\#_{r_1}, \#_{r_2}, \#_{r_3}}(A,B,C).$$
\end{corollary}

The above three results concern the properties of  $M_{\#,\#,\#}$ 
on a finite-dimensional Hilbert space. It is conjectured that these result also
 hold without the assumption of finite dimensionality.
Indeed, for $M_{\#,\#,\#}$, this conjecture, namely self-adjointness, holds.

\begin{proposition}\label{Self-adjoint}
$$(M_{\#,\#,\#})^*=M_{\#,\#,\#}.$$
\end{proposition}

\bigskip
\begin{lemma}
If $A\ge B \ge C>0$, then
$$
M_{\#,\#,\#}(A^{-1},B^{-1},C^{-1})=M_{\#,\#,\#}(A,B,C)^{-1}.$$
\end{lemma}

\begin{proof}
We continue to use the notations $A_n$ and $[A^{-1}]_n$ as in the proof of 
Proposition \ref{matrix adjoint}. 
It follows from Theorem \ref{ordered} that 
$A_{2n} \downarrow M(:=M_{\#,\#,\#}(A,B,C))$ and 
$[A^{-1}]_{2n} \uparrow  M_{\#,\#,\#}(A^{-1},B^{-1},C^{-1})$. 
Thus we have 
\begin{align*}
\|M^{-1}x- [A^{-1}]_{2n}x\|
&=
\|(M^{-1}- A_{2n}^{-1})x\| \\
&=
\|A_{2n}^{-1}(A_{2n}-M)M^{-1}x\| \\
&\le 
\|A_{2n}^{-1}\| \|(A_{2n}-M)M^{-1}x\| \\
&\le
\|C^{-1}\|\|(A_{2n}-M)M^{-1}x\| \rightarrow 0\ \  (n\rightarrow \infty).  
\end{align*}
This shows the desired result. 
\end{proof}

\begin{proof}[Proof of Proposition \ref{Self-adjoint}]
From the discussion in Section~\ref{adjoint of mean}, it is sufficient to show that
\[
M_{\#,\#,\#}(A^{-1}, B^{-1}, C^{-1}) = M_{\#,\#,\#}(A, B, C)^{-1}
\]
for any positive invertible operators \( A, B, C \).  
In this case, there exist scalars \( \lambda, \mu > 0 \) such that
\[
\lambda A \ge \mu B \ge C > 0.
\]
Then, by applying the above lemma and Proposition~\ref{joint homo}, we obtain
\begin{align*}
\lambda^{-1/3}\mu^{-1/3}M_{\#,\#,\#}(A, B, C)^{-1}
&=
M_{\#,\#,\#}(\lambda A, \mu B, C)^{-1} \\
&=
M_{\#,\#,\#}((\lambda A)^{-1}, (\mu B)^{-1}, C^{-1}) \\
&=
M_{\#,\#,\#}(\lambda^{-1} A^{-1}, \mu^{-1} B^{-1}, C^{-1}) \\
&=
\lambda^{-1/3} \mu^{-1/3} M_{\#,\#,\#}(A^{-1}, B^{-1}, C^{-1}).
\end{align*}
\end{proof}
{\Remark
Proposition \ref{Self-adjoint} is a result that has already been acknowledged in other papers (cf. \cite{JLLY}),   however, the authors have not seen a proof of it. Here, as an application of the preceding theorem, we provide a proof in the hope that it will be of benefit to the reader.
It would be of interest to know whether this argument is essentially the same as the one given elsewhere, if available.
}
\bigskip

\begin{corollary}
Let $A,B,C >0$ and $A^{(k)}, B^{(k)}, C^{(k)}$ be sequences of invertible positive operators  
such that $A^{(k)}\uparrow A$,\ \ $B^{(k)}\uparrow B$,\ \ \ $C^{(k)}\uparrow C$. 
Then 
$$ M_{\#,\#,\#}(A^{(k)},B^{(k)},C^{(k)})\uparrow 
M_{\#,\#,\#}(A,B,C).$$
\end{corollary}
\begin{proof}
Note that 
$(A^{(k)})^{-1}\downarrow A^{-1}$, 
$(B^{(k)})^{-1}\downarrow B^{-1}$, 
$(C^{(k)})^{-1}\downarrow C^{-1}$. 
Hence, by Proposition \ref{Self-adjoint} and 
Proposition \ref{downward}, we have 
\begin{align*}
&\lim_k M_{\#,\#,\#}((A^{(k)}),(B^{(k)}),(C^{(k)})) v \\
&=
\lim_k M_{\#,\#,\#}((A^{(k)})^{-1},(B^{(k)})^{-1},(C^{(k)})^{-1})^{-1}v \\
&=
M_{\#,\#,\#}(A^{-1},B^{-1},C^{-1})^{-1}v \\
&=
M_{\#,\#,\#}(A,B,C)v, 
\end{align*}
for all $v\in {\cal H}$.
\end{proof}



\section*{Acknowledgements}
We are grateful to Professor Fumio Hiai for his valuable suggestions during the preparation of this manuscript.
We also wish to thank Professors Takeaki Yamazaki and Shigeru Furuichi for their stimulating and insightful comments.

\section*{Funding} 
This work was supported by JSPS KAKENHI Grant Number JP23K
03141.

\vfill
\section*{Author Information}

\noindent
\textbf{Dante Hoshina} \\
Advanced Course of Control and Information Engineering\\
Kisarazu National College of Technology \\
Email: \href{mailto:sdj24b11@inc.kisarazu.ac.jp}{sdj24b11@inc.kisarazu.ac.jp}

\vspace{1em}

\noindent
\textbf{Shuhei Wada} \\
Department of Information and Computer Engineering,\\
Kisarazu National College of Technology \\
Email: \href{mailto:wada@j.kisarazu.ac.jp}{wada@j.kisarazu.ac.jp}
\end{document}